\documentclass[11pt]{amsart}
\numberwithin{equation}{section}

\usepackage{amsmath,amsthm,amssymb,array,latexsym,amsthm,pdfsync,hyperref,xcolor}
\usepackage{color}
\theoremstyle{plain}
\newtheorem{mydef}{Definition}

\newcommand{\cl}[1]{\mathcal{#1}}
\newtheorem{thm}{Theorem}
\newtheorem{lem}{Lemma}
\newtheorem{cor}{Corollary}
\newtheorem{prop}{Proposition}

\newtheorem{examp}{Example}
\theoremstyle{remark}
\theoremstyle{remark}\newtheorem{rem}{Remark}
\theoremstyle{remark}\newtheorem{notat}{Notation}

\def\n{\noindent}
\def\Ex{\mathbb{E}}

\def\notin{\epsilon \hspace{-0.37em}/}

  \def\rd{\color{red}}

\newcommand{\E}{\ensuremath{\mathbb E}}
\date{}

\begin{document}

\title{Concerns with Functional Depth}
\author{James Kuelbs} 
\address{James Kuelbs\\Department of Mathematics,  University of Wisconsin, Madison, WI 53706-1388}
\email{kuelbs@math.wisc.edu}
\author{Joel Zinn}
\address{\noindent Joel Zinn\\Department of Mathematics, Texas A\&M University, College Station, TX 77843-3368}
\email{jzinn@math.tamu.edu}

\begin{abstract}  We study some problems inherent with certain forms of functional depth, in particular, zero depth and lack of consistency. 
\end{abstract}

\maketitle

\section{Introduction}

The use of a variety of depth functions to provide a center-outward ordering of finite dimensional data is well established, and more recently a number of interesting papers have considered analogues of such depths to study multivariate data in the infinite dimensional setting. These depths apply to data given in terms of infinite sequences, as  functions defined on some interval, and also as points in some infinite dimensional  Banach space. The papers \cite{lp-r-concept}, \cite{lp-r-half}, \cite{dutta-tukey}, \cite{chak-chaud-12}, \cite{mosler-poly-funct-depth}, and \cite{cuevas-fraiman-dual} introduce interesting  examples of such depths, and also include many additional references, but the focus here centers on the problem that the natural analogue of some commonly used depths in $\mathbb{R}^d$ may well be zero "most of the time" in the infinite dimensional setting. By "most of the time" it is meant that the depth is zero on a set whose probability is one with respect to 
 the prob
 ability the depth is based on. This was pointed out in Theorem 3 of \cite{dutta-tukey} for Tukey's half-space depth with respect to certain probability laws on the Hilbert space $\ell_2$, and also for the band depth and half-region depth of \cite{lp-r-concept}, \cite{lp-r-half} in \cite{chak-chaud-12}. The paper \cite{kz-half-region} also examined this problem for half-region depth showing it not only vanishes most of the time, but in many examples it vanishes everywhere. Moreover, in \cite{kz-half-region} it is also shown how one can smooth such data so as to regain positive half-region depth, and then establish consistency for the empirical half-region depth of the smoothed data. In some cases one can also show $\sqrt{n}$ consistency.  In a related, but slightly different context, smoothing was used to guarantee various limit theorems (see \cite{kkz}, \cite{kz-quant}).  
\bigskip

In the present work our goal is to better understand both the problem of zero depth as well as questions of consistency for infinite dimensional data. Our results are in the context of Tukey half-space depth, and also for the simplicial depth in the infinite dimensional setting obtained as in \cite{mosler-poly-funct-depth}. Of course, there are other possible choices, but the importance of our choices in the finite dimensional setting made them attractive to study for functional data.  Moreover, it is clear that for data symmetric about zero, the half-space depth of the zero vector will typically be $\frac{1}{2}$, but the result in \cite{dutta-tukey} shows this depth may also be zero with probability one. Hence an immediate question to ask is for what other points might this depth be strictly positive? Also, perhaps the set of points at which the depth is non-zero is ``special'' and  one can prove consistency for these special points. 

In this paper, in a number of cases, we'll describe the precise set of points at which the half-space depth is strictly positive and when it is zero. This is accomplished via Theorems \ref{necessity-zero-depth} and \ref{nasc-indep} of section \ref{inf-dim}, and we also provide some explicit formulas for half-space depth in special cases in section \ref{examples}. Theorem \ref{emp-depth} of section \ref{inf-dim} shows that in many situations the empirical half-space depth is zero with probability one. Combined with Theorems \ref{necessity-zero-depth} and \ref{nasc-indep}, this last result is particularly bad news, as often the empirical depth is zero with probability one at all the points where the true depth is strictly positive. Of course, if the true half-space depth is zero, in this case  we would have consistency, but very little information. Also, the remark following the statement of Theorem 3 below points out several aspects of such probl
 ems, and
  how the results 
 of Theorem 3 differ from those on Tukey functional depth on  page  eleven of \cite{mosler-poly-funct-depth}.

Proposition \ref{gaussian} of section \ref{examples}  shows how these  results combine for Gaussian measures with infinite dimensional support on a separable Banach space, and, although we can identify the precise set of positivity of the half space depth based on $B^{*}$, there 
still is a lack of consistency. In the last part of section \ref{examples} we see the simplicial depth of Liu (see \cite{liu-simplicial-88} and \cite{liu}) extended to $\mathbb{R}^{\infty}$ using definition one in section 7 of \cite{mosler-poly-funct-depth} is also subject to the very same problems. This is an interesting fact in its own right, but also because this depth is quite different (see \cite{serfling-zuo-notions}) than Tukey's half-space depth. It is also interesting to note that the context in which zero empirical depth appears for these particular depths seems to occur when there are a lot of independence-like properties in the data. Hence, it may be possible to modify or smooth either the random variable or the depth to ensure consistency or even a central limit type theorem, but that is off in the future.
 
There are several other interesting aspects of this study that should be mentioned. These include the fact that the results of Theorem \ref{nasc-indep} connect with admissible translates of probability measures on $\mathbb{R}^{\infty}$ (see \cite{kakutani-prod-equiv}, \cite{shepp-admissible}), and in subsection \ref{rad-case}  we need some delicate tail estimates of Rademacher series to estimate the magnitude of the actual half-space depth at a point (see \cite{monty-rad-sums}). 

Finally, it seems to be clear that each example of functional depth brings with it its own difficulties. Some may be more immune to 
various difficulties than others. This can be seen in the following example of the dual integrated depth of Cuevas and Fraiman \cite{cuevas-fraiman-dual}, where positivity of the depth holds, and we also have an immediate link to positivity of the modified band depth of Lopez-Pintado and Romo \cite{lp-r-concept}. We start with the one-dimensional version of the band depth of Lopez-Pintado and Romo \cite{lp-r-concept}. That is, if $r\ge 2$ and for $\{\xi_{j}\}_{j=1}^{r}$ iid with distribution $\mu$,
\[BD_{r,1}(b,\mu)=P(\min_{1\le j\le r}\xi_{j}\le b\le \max_{1\le j\le r}\xi_{j}).
\]
Of course, $BD_{2,1}$
is the univariate version of the simplicial depth,  and 
$BD_{2,1}(b,\mu)\le BD_{r,1}(b,\mu)$.
Hence, it is positive provided $b$ is in the interior of the closed convex hull of the support of $\mu$, or at a boundary point, if the boundary point has positive $\mu$-probability.
Assume that $Q$ is  the measure on the set of point evaluations, $e_{t}(a)=a(t)$, say for $a$ in $C[0,1],$ given by Lebesgue measure, $m$, on $[0,1]$. (For $A$ a Borel subset in the weak-star topology of $C^{*}[0,1]$, the dual space of $C[0,1]$, define $Q(A)=m(t\in[0,1]\colon e_{t}\in A)$). Now, consider a stochastic process, $\{X(t): t\in [0,1]\}$ with distribution, $P$, on $C[0,1]$ and i.i.d. copies, $\{X_{j}(t): t\in [0,1]\}_{j=1}^{\infty}$, and a function, $a \in C[0,1]$. Then the definition of the 
dual integrated depth associated with the depth $BD_{r,1}$ gives 
\begin{align}\label{dbd-mbd}&IDD(a,P)=\int_{0}^{1}P(\min_{1\le j\le r}X_{j}(t)\le a(t)\le \max_{1\le j\le r}X_{j}(t))\ dt\\
&=(\text{by Fubini})\ \Ex\ [m(t\in [0,1]: \min_{1\le j\le r}X_{j}(t)\le a(t)\le \max_{1\le j\le r}X_{j}(t))]\notag\\
&= MBD_{r,1}(a, P), \notag
\end{align}
where $MBD_{r,1}(a, P)$ is the modified band depth of Lopez-Pintado and Romo \cite{lp-r-concept}.
In particular, the positivity of the integrand in (\ref{dbd-mbd}) on some open subinterval of $[0,1]$ implies that for each such $a$, the quantities in (\ref{dbd-mbd}) are positive. For example, if $s_{1,t} = \inf\{x:P(X(t) \le x)>0\}$ and $s_{2,t}=\sup\{x:P(X(t) \le x)<1\}$ for $t \in [0,1]$, and $E=\{t \in [0,1]: s_{1,t}= -\infty, s_{2,t}= \infty\}$ has positive Lebesgue measure, then for all $a \in C[0,1]$ and $t \in E$ the quantity
$$
P(\min_{1\le j\le r}X_{j}(t)\le a(t)\le \max_{1\le j\le r}X_{j}(t))
$$
is strictly positive. Also, one should see Theorem 2 in \cite{cuevas-fraiman-dual} on the consistency of IDD and  \cite{lp-r-concept} for consistency results for the (unmodified) band depth.

\section{ Infinite dimensional half space depth.} \label{inf-dim}

Here we formulate some results on half space depth in infinite dimensional, real, topological vector spaces $B$, whose topology is metrizable, complete and separable via a translation invariant metric. They are the so-called F-spaces in \cite{rudin}, and include the real separable Banach spaces, as well as Fr\'echet spaces such as $\mathbb{R}^{\infty},$ and many other topological vector spaces.
\bigskip

Throughout $X,X_1,X_2,\ldots$ are i.i.d. $B$-valued random vectors on the probability space $(\Omega, \mathcal{F},P)$ which are measurable from $\mathcal{F}$ to the Borel sets $\mathcal{B}_B$ of $B$, and 
$\mu$ denotes the law of $X$ on $(B, \mathcal{B}_B)$.  We also assume $\mathcal{T}$ is a collection of  Borel measurable functionals on $B$. Then, we define the half space depth of $a \in B$ with respect to $\mathcal{T}$ and $\mu$ to be
\begin{align}\label{hsd}
HD_{\mathcal{T}}(a,\mu)=\inf_{t \in \mathcal{T}} P(\omega \in \Omega: t(X(\omega)) \ge t(a)). 
\end{align}

We usually denote the right hand term by $\inf_{t \in \mathcal{T}} P( t(X) \ge t(a))$, and observe that we also have
\begin{align}
HD_{\mathcal{T}}(a,\mu)= \inf_{t \in \mathcal{T}} \mu(x \in B: t(x) \ge t(a)). 
\end{align}
If $\mu_{n}(\omega)= \frac{1}{n} \sum_{j=1}^n \delta_{X_j(\omega)}, n \ge 1,$ then for each $\omega \in \Omega$  we have $\mu_{n}$ a probability measure on $B$ and the empirical half space depth for $a \in B$ with respect to $\mathcal{T}$ and $\mu_{n}$ is defined to be
\begin{align}\label{ehsd}
HD_{\mathcal{T}}(a,\mu_{n})= \inf_{t \in \mathcal{T}} \mu_{n}(x \in B: t(x) \ge t(a)),  
\end{align}
and hence we also have
\begin{align}\label{ehsd'}
HD_{\mathcal{T}}(a,\mu_{n})= \inf_{t \in \mathcal{T}}\frac{1}{n} \sum_{j=1}^n I(t(X_j) \ge t(a)).  
\end{align}

If $B = \mathbb{R}^d$ and $\mathcal{T}$ is the linear functionals on $\mathbb{R}^d$ in (\ref{hsd}), this is Tukey half space depth. Of course, since the linear functionals on $\mathbb{R}^d$ are given by inner products, we denote this by $HD_{\mathbb{R}^d}(a,\mu),$ and observe that 
\begin{align*} 
HD_{\mathbb{R}^d}(a,\mu)=  \inf_{t \in \mathbb{R}^d} P( t(X) \ge t(a))= \inf_{t \in \mathbb{R}^d } \mu(x \in \mathbb{R}^d: t(x) \ge t(a)).
\end{align*}
When $B$ is a real separable Banach or Fr\'echet space, with the dual space of continuous linear functionals on $B$ denoted by $B^*,$ then a natural definition of Tukey half space depth with  $\mathcal{T} = B^{*}$
is given by 
\begin{align*} 
HD_{B^*}(a,\mu)= \inf_{t \in B^{*}} P( t(X) \ge t(a)). 
\end{align*}
It should be observed that in the literature the half space depth we defined on $\mathbb{R}^d$ or $B$ is likely to be written as $HD_\Ex(a,\mu)$ where $E$ is  $\mathbb{R}^d$ or $B$, respectively. We chose our terminology to emphasize that the functionals in $\mathcal{T}$ need not be continuous or linear on $B$.
\bigskip

The point to be seen here is that in the infinite dimensional setting the class $B^{*}$ is frequently much too large to provide positive depth on much of the space $B$. Moreover, problems of consistency emerge even if the half space depth is positive at a point and $\mathcal{T}$ is countably infinite. Hence it is useful to formulate depth as in (\ref{hsd}) where the class of functionals $\mathcal{T}$ allows more flexibility.
It is also important to note that the functionals $t \in \mathcal{T}$ need not be linear or continuous on $B$, and there are good reasons for this. First, in the generality we are considering there are examples, such as $L^{p}$ for $0<p<1$, where the only continuous linear functional on 
$B$ will be the functional that is identically zero,  and of more importance, the functionals of interest need not be linear to start with. For example, in section {\rd \ref{sec:MP}} we show a lack of consistency for the Mosler-Polyakova version of Liu's simplicial depth using maps that are probabilities (which are highly non-linear functions of the data).
\bigskip

\begin{examp} To see the effect that different choices of $\mathcal{T}$ can have on half-space depth, let $X$ take values in $C[0,1]$, the space of continuous functions on $[0,1]$, where $\mathcal{L}(X(s))=N(0,1+s)$ and $E(X(s)X(t))=1+\min\{s,t\}$ for $s,t \in (0,1)$.
Then, $\{X(s): 0 \le s \le 1\}$ is a standard sample continuous Brownian motion started randomly at time zero with a $N(0,1)$ distribution, and if $\mathcal{T}$ consists of the evaluation maps
$$
e_{t}(a)=a(t), t \in [0,1], a \in C[0,1],
$$
we have
$$
HD(a,\mu) =\inf_{t \in [0,1]}P(X(t) \ge a(t)),
$$
where $\mu$ is the law of $X$ on $C[0,1]$.
Therefore,
$$
HD(a,\mu) =\inf_{t \in [0,1]}P(\frac{X(t)}{\sqrt{1+t}} \ge \frac{a(t)}{\sqrt{1+t}})=\inf_{t \in [0,1]}[1-\Phi( \frac{a(t)}{\sqrt{1+t}})],
$$
and since $\sup_{t \in [0,1]}\frac{|a(t)|}{\sqrt{1+t}} < \infty$ for all $a \in C[0,1]$ we have
$HD(a,\mu) >0$
for all $a \in C[0,1]$.
On the other hand, if $\mathcal{T}$ consists of the differences of two evaluation maps, then as we will now see the half-space depth will be zero with $\mu$ probability one. In fact, we need not consider all differences, but only that $\mathcal{T}$ consists of the sequence of differences   
$$
\theta_k(a)\equiv a(\frac{1}{k})-a(\frac{1}{k+1}), k \ge 1.
$$
Then, the half-space depth of a function $a \in C[0,1]$ with respect to $\mu$ and this choice of $\mathcal{T}$ is
$$
HD(a,\mu) =\inf_{k \ge 1}P(\theta_k(X) \ge \theta_k(a)).
$$
Now $G_k=\sqrt {k(k+1)}\theta_k(X), k \ge 1,$ are i.i.d. $N(0,1)$ random variables and hence for $a \in C[0,1]$
$$
HD(a,\mu) =\inf_{k \ge 1}P(G_k \ge \sqrt{k(k+1)}\theta_k(a))= 1-  \Phi(\sup_{k \ge 1} \sqrt{k(k+1)}\theta_k(a)).
$$
If 
$$
A= \{ a \in C[0,1]:  \sup_{k \ge 1}\sqrt{k(k+1)}\theta_k(a)=\infty\},
$$
then $\mu(A) = P(\sup_{k \ge 1}G_{k})=1$, and hence with $\mu$-probability one $HD(a,\mu) =0$ when $\mathcal{T}$ consists of these differences.
\end{examp}

Now we turn to three theorems and the additional notation used in the remainder of the paper. The first theorem obtains sufficient conditions for half space depth to be zero in the infinite dimensional setting, and following its statement there are a couple of remarks indicating how Theorem 3 of \cite{dutta-tukey} for half space depth in $\ell_2$ follows as a special case. These remarks also examine other aspects of the theorem.  A second theorem, when combined with Theorem \ref{necessity-zero-depth}, establishes necessary and sufficient conditions that the depth be positive, and the third examines when the empirical version of this depth in (\ref{ehsd}) and (\ref{ehsd'}) approximates the true distributional depth. There are also corollaries and remarks pertaining to these results, which indicate 
how they fit together. The proofs of the theorems appear at the end of the section.

\begin{thm}\label{necessity-zero-depth}  Let $X$ be a random vector with values in a real separable F-space $B$, and assume the
functionals $\{t_k: k\ge 1\}$ are measurable from the Borel subsets of $B$ to the reals. In addition, assume $\{t_k(X): k \ge 1\}$ are mean zero random variables with variances $\sigma_k^2 \in (0, \infty), k \ge 1,$ 
and that $\Ex(t_i(X)t_j(X))=0, 1 \le i<j<\infty$. If $\mathcal{T}$ consists of all finite linear combinations of the maps $\{t_k: k \ge 1\}$, then  
\begin{align}\label{thm1-conclusion}
HD_{\mathcal{T}}(a,\mu)= 0
\end{align}
for all $a \in B$ such that
\begin{align}\label{thm1series-hyp}
\sum_{k \ge 1} \frac{t_k^2(a)}{\sigma_k^2} = \infty. 
\end{align}
Furthermore, (\ref{thm1-conclusion}) holds with $\mu$ probability one if 
\begin{align}\label{thm1-weak-series-hyp}
P(\sum_{k\ge1} t_k^2( X)/\sigma_k^2 =\infty)= 1.
\end{align}
\end{thm}

Since $\mathcal{T}$ contains the finite linear combinations of the $\{t_k: k \ge 1\}$ it will be convenient to parameterize these functionals  by letting $\ell_0$
denote the sequences in $\mathbb{R}^{\infty}$ which have only finitely many non-zero terms, and for $\alpha=(\alpha_1,\alpha_2,\ldots) \in \ell_0$ let
\begin{align}\label{talpha}
t_{\alpha} = \sum_{k\ge 1} \alpha_kt_k. 
\end{align}
Then, the right hand term in (\ref{talpha}) is a finite sum, $t_{\alpha}$ is a typical functional in $\mathcal{T}$, and the half space depth of $a \in B$ with respect to $\mathcal{T}$ and $\mu$ satisfies
\begin{align}\label{hsdell0}
HD_{\mathcal{T}}(a,\mu)= \inf_{\alpha \in \ell_0} P( t_{\alpha}(X) \ge t_{\alpha}(a)). 
\end{align}
\bigskip

\begin{rem} Let $\{t_k(X): k\ge 1\}$ be as in Theorem \ref{necessity-zero-depth}. Hence, if for some sequence $\{a_n:n\geq 1\}$ increasing to infinity, we have 
\begin{align}\label{stability}
\limsup_{n \rightarrow \infty} \frac{1}{a_n}\sum_{k=1}^n t_k^2(X)/\sigma_k^2 >0
\end{align}
with probability one, then (\ref{thm1-weak-series-hyp}) holds and the final conclusion of Theorem 1 implies
 (\ref{thm1-conclusion}) with $\mu$-probability one. Futhermore, since $\Ex(t_k^2(X)/\sigma_k^2)=1$ for $k \ge1,$ if $a_n=n$ the stability result in (\ref{stability}) would immediately hold from the ergodic theorem if the sequence $\{t_k^2(X)/\sigma_k^2\}$ is stationary and ergodic. It also follows without the ergodicity assumption provided we have stationarity and $P(t_1(X)=0) =0$. Of course, if the random variables $\{t_k(X): k\ge 1\}$ are assumed independent, then (\ref{stability}) holds  with $a_n=n$ and limit one in a variety of situations by applying a law of large numbers. For example, under the independence assumption and that 
\begin{align}\label{duttaell4}
\sum_{k\ge1}\frac{\Ex(t_k^4(X))}{k^2 \sigma_k^4} < \infty, 
\end{align} 
this is the case. However, the condition (\ref{stability}) also follows with $a_n=n$ and the limit being one provided (\ref{duttaell4}) holds and that $\{t_k^2(X)/\sigma_k^2: k \ge 1\}$ are uncorrelated. That is, under these conditions 
it is easy to check that
\begin{align}\label{wlln-stationary}
\Ex([\frac{1}{n}\sum_{k=1}^n\frac{t_k^2(X)}{\sigma_k^2} - 1]^2) =\frac{1}{n^2} \sum_{k=1}^n \Ex(\frac{t_k^4(X)}{\sigma_k^4}) - \frac{1}{n}. 
\end{align}
Hence, (\ref{duttaell4}) and Kronecker's Lemma combine to imply
\begin{align}\label{stability^{2}}
\lim_{n \rightarrow \infty}\frac{1}{n^2} \sum_{k=1}^n \Ex(\frac{t_k^4(X)}{\sigma_k^4})=0.
\end{align}
Therefore,
$\frac{1}{n}\sum_{k=1}^n\frac{t_k^2(X)}{\sigma_k^2}$ converges in $L^2$  to one, and (\ref{stability}) holds with $a_n=n$ and limit one with probability one.
\end{rem}
\begin{rem} If $X$ takes values in the real separable Banach space $\ell_2$, then Theorem 3 of \cite{dutta-tukey} shows that under certain conditions on the distribution of $X$ the Tukey half space depth with $\mathcal{T} = \ell_2^{*}$ is zero. In that result the maps $t_k(X)$ are assumed to be mean zero independent random variables  such that $t_k(X)= \langle X,e_k\rangle$,where $\langle \cdot,\cdot \rangle$ is the inner product on $\ell_2$ and $\{e_k: k \ge 1\}$ is the canonical basis of $\ell_2$. Furthermore, (\ref{duttaell4}) is assumed to hold, and the half space depth is defined in terms of all continuous linear functionals on $\ell_2$. In our terminology, the finite linear combinations of the $\{t_k(X): k \ge 1\}$ we denote by $\{t_{\alpha}: \alpha \in \ell_0\}$ would be replaced by $\{t_{\alpha}: \alpha \in \ell_2\}$, and hence that depth is less than or equal the depth we use. Since zero is the minimal possible depth, our result in Proposition 1 therefore implies the
  result in \cite{dutta-tukey}. Moreover, it implies similar results in any F-space $B$ without an independence or a linearity assumption on the mappings $\{t_k: k\ge 1\}$.
\end{rem}

If $X$ is symmetric about the vector $a \in B$ and the maps $\{t_k: k \ge 1\}$ are linear, then for each $\alpha \in \ell_0$ we have $P(t_{\alpha}(X) \ge t_{\alpha}(a))\ge 1/2$, and hence 
$HD_{\mathcal{T}}(a,\mu)  \ge 1/2$. Furthermore, it will equal $1/2$ if $P(t_k(X)=t_k(a))=0$ for all $k \ge 1$. Thus  certain vectors have positive half space depth, and our next proposition examines for which vectors in $B$ this might be the case. However, in order to provide sufficient conditions for positive half space depth we require some additional assumptions. They are:
\bigskip

{\bf Assumptions.}
\noindent (A-I) For $a \in B$ 
and all integers $d \ge 1$ 
\begin{align}\label{A-I(i)}
HD_{\mathbb{R}^d}(\Pi_d(a),\mu^{\Pi_d})>0,  
\end{align}
where $\Pi_d(a)=(t_1(a),\ldots,t_d(a))$ and $\mu^{\Pi_d}$ is the image of $\mu$ on $\mathbb{R}^d$ via the map $\Pi_d(\cdot): B \rightarrow \mathbb{R}^d$,
\par \bigskip
\noindent (A-II) For some constant $c < \infty$, $\Ex(t_k^4(X))\le  c[\Ex(t_k^2(X))]^2$ for all $k \ge 1,$ 

\par\bigskip
\noindent and
\bigskip
 
\noindent (A-III) $\{t_k(X)/\lambda_k: k \ge 1\}$, $0<\lambda_k<\infty,$ are i.i.d. with probability density $\phi$ that is positive a.s., (locally) absolutely continuous on $\mathbb{R}$ (i.e., $\phi$ is absolutely continuous on every compact interval of $\mathbb{R}$) satisfying 
\begin{align}\label{FisherInformation}
\mathcal{I}(\phi)=\int_{\mathbb{R}} \frac{(\phi')^2(x)}{\phi(x)}dx < \infty, 
\end{align}
and 
\begin{align}\label{phi-var}
\sigma^2 = \int_{\mathbb{R}} x^2 \phi(x)dx < \infty. 
\end{align}
\bigskip

\begin{rem}  The condition (\ref{A-I(i)}) may be difficult to check in some situations, but it is a necessary condition for $HD_{\mathcal{T}}(a,\mu)$ to be strictly positive since 
\begin{align*}
HD_{\mathbb{R}^d}(\Pi_d(a),P^{\Pi_d})= \inf_{ \alpha \in \ell_0} P(\sum_{k=1}^d \alpha_kt_k(X) \ge \sum_{k=1}^d \alpha_kt_k(a)) 
\end{align*}
\begin{align*}
\ge \inf_{\alpha \in \ell_0} P(\sum_{k \ge 1}\alpha_k t_k(X) \ge \sum_{k \ge 1}\alpha_k t_k(a))=HD_{\mathcal{T}}(a,\mu).
\end{align*}

\n The quantity in (\ref{FisherInformation}) is often called the Fisher information. It appeared in \cite{Fisher73} and was used in \cite{shepp-admissible} in connection with admissible translates (see Definition \ref{AdmTrans} below).  While the uses of the Fisher information are ubiquitous in Statistics, at this point we only use the connection to admissible translates. Furthermore, if $\mu$ is a probability measure on $\mathbb{R}^d$ with probability density that is strictly positive a.s. with respect to Lebesgue measure, then every vector $b \in \mathbb{R}^d$ is an admissible translate. Therefore, by the proof of Lemma \ref{move-R-infty} every vector in $\mathbb{R}^d$ has positive half space depth with respect to $\mu$ for $\mu$ symmetric, and under these conditions (\ref{A-I(i)}) holds.
In addition, the conclusion of Theorem \ref{nasc-indep} given by the assumptions in (A-III) would then follow from the assumptions in (A-I) and (A-II) provided we also assume $\int_{\mathbb{R}} x^4\phi(x)dx<\infty.$
Of course, the conditions in (A-I) and (A-II) do not require that $\{t_k(X): k \ge 1\}$ have densities, and they apply without the  $\{t_k(X): k \ge 1\}$ being scaled  i.i.d. random variables. Hence, in that sense they are more general than what can be obtained from the assumptions in (A-III), but it is also of interest that the conditions in (A-III) yield results without a fourth moment assumption as in (A-II).
\end{rem}

Condition (\ref{FisherInformation}) and Lemma \ref{move-R-infty} below allow us to link half space depth to admissible translates of product measures and the results of Kakutani \cite{kakutani-prod-equiv} and Shepp \cite{shepp-admissible}. 
\begin{thm}\label{nasc-indep} Let $X$ be a random vector with values in a real separable F-space $B$, and assume $\{t_k(X): k \ge 1\}$ are independent, symmetric random variables with variances $\sigma_k^2 \in (0, \infty), k \ge 1.$  In addition, assume (A-I) and (A-II) hold, or (A-III) holds. Then, for $a\in B$ and  $\mathcal{T}$ all finite linear combinations of the maps $\{t_k: k \ge 1\}$,
\begin{align}\label{thm2hd>0}
HD_{\mathcal{T}}(a,\mu)> 0     
\end{align}
if and only if 
\begin{align}\label{series-for-a}
\sum_{k\ge 1} \frac{t_k^2(a)}{\sigma_k^2} < \infty. \
\end{align}
\end{thm}

Our next theorem examines empirical half space depth, and the corollary and remark following its statement clarify consistency for the empirical depth in the setting of
Theorems \ref{necessity-zero-depth} and \ref{nasc-indep} provided (\ref{thm3ReqCond}) holds.
\bigskip

\begin{thm}\label{emp-depth} Let  $\{X_j: j \ge 1\}$ be i.i.d. copies of $X$ where $X$ is a random vector taking values in a real F-space $B$ and 
$\{t_k(X): k \ge 1\}$ are independent mean zero random variables
with variances $\sigma_k^2 \in (0, \infty), k \ge 1.$ Furthermore, assume for all  
$c_k,k \ge 1,$ 
such that $\lim_{k \rightarrow \infty} c_{k}/\sigma_{k}=0,$
\begin{align}\label{thm3-assump}
\liminf_{k \rightarrow \infty} P(t_{k}(X)< c_{k})>0, 
\end{align}
and $\mathcal{T} \supseteq \{t_k(\cdot): k \ge 1\}$.
Then, for all $ n \ge 1$ and $a \in B$ with 
\begin{align}\label{thm3ReqCond}
\lim_{k \rightarrow \infty} \frac{t_k(a)}{\sigma_k}=0, 
\end{align}
the empirical half space depth
\begin{align}\label{2.21}
HD_{\mathcal{T}}(a,\mu_{n})\equiv \inf_{t \in \mathcal{T}} \frac{1}{n} \sum_{j=1}^n I(t(X_j) \ge t(a))=0 
\end{align}
with probability one.
\end{thm}

\begin{rem} The empirical half space depth, $HD_{\mathcal{T}}(a,\mu_{n})$, can be thought to have two "random quantities" on different probability spaces. One, say, $\omega$ with corresponding probability, $P$,  is through the random variables in $\mu_{n}$ and the other is $a$ with respect to the induced measure, $\mu$,  on the space, $B$.  In \cite{mosler-poly-funct-depth} Mosler and Polyakova give a result that holds quite generally for Tukey-like depths as long as $\mathcal{T}=B^{*}$. Namely that fixing $\omega$ and computing with respect to $\mu$ the empirical depth equals zero with $\mu$-probability one provided finite dimensional subspaces have $\mu$ probability zero.
On the other hand, in special, but interesting circumstances (as in the above Theorem \ref{emp-depth}), we fix $a$ in a large class and show, that the empirical depths are zero with $P$-probability one. So, in particular, we can say for each point in this ``large class'' that consistency fails or fails to give any information.
\end{rem}

\begin{cor}\label{consist} Assume the conditions in Theorem \ref{emp-depth} with $\mathcal{T}$ all finite linear combinations of $\{t_k(\cdot): k \ge 1\}$, and that (A-I) and (A-II) hold, or (A-III) holds. Then, for each $a \in B$ such that 
\begin{align}\label{series-for-a-again}
\sum_{k \ge 1}\frac{t_k^2(a)}{\sigma_k^2}<\infty, 
\end{align}
the empirical half space depth 
\begin{align}\label{cor1EmpHS=0}
HD_{\mathcal{T}}(a,\mu_{n})=0 
\end{align}
with probability one for all $n\ge 1$, and the half space depth $HD_{\mathcal{T}}(a,\mu)> 0$. Hence the natural  empirical half depth fails to approximate the true half space depth at such points $a \in B$, i.e. consistency fails at all such points. However, we do have consistency for those $a\in B$ where (\ref{series-for-a}) fails, but (\ref{thm3ReqCond}) holds.
\end{cor}

\begin{rem} The proof of Corollary \ref{consist} follows immediately from Theorems  \ref{necessity-zero-depth}, \ref{nasc-indep} and \ref{emp-depth}. Furthermore, if the condition in (\ref{thm3-assump}) is replaced by the assumption 
\begin{align*}
\liminf_{k \rightarrow \infty} P(t_{k}(X)< t_{k}(a))>0,
\end{align*}
then the proof of Theorem \ref{emp-depth} implies that  $HD_{\mathcal{T}}(a,\mu_n)=0$ with probability one for all $n \ge 1$ without assuming the variances $\sigma_k^2$ exist. However, the condition (\ref{thm3ReqCond}) allows us to relate empirical half space depth to the true half space depth as indicated in Corollary \ref{consist}. Moreover, it is only under the assumptions in (A-I) and (A-II) where the condition (\ref{thm3-assump}) is something extra. That is, the assumptions in (A-III) imply (\ref{thm3-assump}).
This can be seen by observing that $P(t_{k}(X)< c_k)= \int_{-\infty}^{\frac{c_k}{\lambda_k}} \phi(x)dx,$ and hence $\sigma_k=\sigma \lambda_k$, $\lim_{k \rightarrow \infty} c_{k}/\sigma_{k}=0$, and $\phi$ symmetric about zero implies $\lim_{k \rightarrow \infty}P(t_{k}(X)< c_k)= 1/2.$ 
\end{rem}

\subsection{Proof of Theorem \ref{necessity-zero-depth}}
 Using (\ref{talpha}) and (\ref{hsdell0})
\begin{align*}
HD_{\mathcal{T}}(a,\mu)\le  \inf_{\alpha \in \ell_0, t_{\alpha}(a)>0} P( t_{\alpha}(X) \ge t_{\alpha}(a)),
\end{align*}
and by Markov's inequality
\begin{align*}
HD_{\mathcal{T}}(a,\mu)\le  \inf_{\alpha \in \ell_0, t_{\alpha}(a) >0} \Ex(t_{\alpha}^2(X)/t_{\alpha}^2(a))= \inf_{\alpha \in \ell_0, t_{ \alpha}(a) >0} \frac{\sum_{k\ge 1} \alpha_k^2 \sigma_k^2 }{(\sum_{k \ge 1} \alpha_k t_k(a))^2} .
\end{align*}
Therefore,
$HD_{\mathcal{T}}(a,\mu) =0$
whenever
\begin{align*}
 \sup_{\alpha \in \ell_0, t_{\alpha}(a) >0} \frac{(\sum_{k \ge 1} \alpha_k t_k(a))^2 }{\sum_{k \ge 1} \alpha_k^2 \sigma_k^2 } =\infty.
\end{align*}
Given $a \in B$ such that (\ref{thm1series-hyp}) holds, then by setting $\alpha_k= t_k(a)/\sigma_k^2, k=1,\cdots,n,$ and zero for $ k\ge n,$ we therefore have
\begin{align*}
 \sup_{\alpha \in \ell_0, t_{\alpha}(a) >0} \frac{(\sum_{k \ge 1} \alpha_k t_k(a)^2 }{\sum_{k\ge 1} \alpha_k^2 \sigma_k^2 } \ge \sup_{n}  \sum_{k= 1}^n t_k(a)^2/\sigma_k^2 = \infty. 
\end{align*}
Thus the theorem is proved as the final assertion that (\ref{thm1-conclusion}) follows from (\ref{thm1-weak-series-hyp}) is now immediate.
\bigskip

\subsection{Proof of Theorem \ref{nasc-indep}}
To prove Theorem \ref{nasc-indep} it will be convenient to first prove some lemmas, where we continue to use the parameterization of $\mathcal{T}$ determined in (\ref{talpha}).\\

\begin{lem}\label{lin-comb-4th-mom} Let $\{t_{k}(X)\colon k\ge 1\}$ be independent with mean zero, $\sigma_{k}^{2}= \E(t_k^2(X))\in(0,\infty),$ and define $\ell_{0}^{+}=\{\alpha\in\ell_{0}: \sum_{k\ge 1}\alpha_{k}^{2}>0\}$. Furthermore, assume there exists $c \in (0,\infty)$ such that  (A-II) holds. 
Then, 
\begin{align}\label{fourthpower}\inf_{\alpha \in \ell_0^+} \frac{[\E((\sum_{k \ge 1}\alpha_k t_k(X))^2)]^2  }{\E((\sum_{k \ge 1}\alpha_k t_k(X))^4)   } \ge  (3 c)^{-1}.
\end{align}
\end{lem}
\begin{proof}  Expanding the sum to the fourth power we have
$$
\E(|\sum_{k\ge 1} \alpha_kt_{k}(X)|^4)= \sum_{k \ge 1}\alpha_k^4\E(t_{k}(X)^4) + 6\sum_{1\le  i < j}\alpha_i^2 \alpha_j^2\E(t_{i}(X)^2)\E(t_{j}(X)^2),
$$
and hence, since $c\ge 1$,
\begin{align*}
\E(|\sum_{k\ge 1} \alpha_kt_{k}(X)|^4) &\le 3c[\sum_{k \ge 1}\alpha_k^4\E^2(t_{k}(X)^2) + 2\sum_{1 \le i<j}\alpha_i^2 \alpha_j^2\E(t_{i}(X)^2)\E(t_{j}(X)^2)]\\
&=3c[\sum_{k \ge 1}\alpha_k^2\E(t_{k}(X)^2)]^2.
\end{align*}
\end{proof}

\begin{lem}\label{using-PZ} If $\{t_k(X): k \ge 1\}$ are independent and symmetric (about zero), (A-II) holds, and $a \in B$ is such that $ \sum_{k \ge 1}\frac{t_k^2(a)}{\sigma_k^2}<1$, then
\begin{align*}
\inf_{\alpha \in \ell_0}P( \sum_{k \ge 1}\alpha_k t_k(X) \ge \sum_{k \ge 1}\alpha_k t_k(a))>0.
\end{align*}
\end{lem}

\begin{proof} Take $\delta>0$ such that $ \sum_{k \ge 1}\frac{t_k^2(a)}{\sigma_k^2}<(1-\delta)^2.$ Then,
\begin{align*}
|\sum_{k\ge 1}\alpha_kt_k(a)| \le ||\{\alpha_k \sigma_k\}||_2||\{t_k(a)/\sigma_k\}||_2  \le (1-\delta)||\{\alpha_k\sigma_k\}||_2,
\end{align*}
and hence 
\begin{align*}
P(\sum_{k \ge 1}\alpha_kt_k(X) \ge \sum_{k\ge 1}\alpha_kt_k(a))\ge P(\sum_{k \ge 1}\alpha_k t_k(X) \ge (1-\delta)||\{\alpha_k\sigma_k\}||_2).
\end{align*}
Since $\{t_k(X): k \ge 1\}$ are independent and symmetric, we thus have
\begin{align}\label{2.25}
P(\sum_{k \ge 1}\alpha_kt_k(X) \ge \sum_{k\ge 1}\alpha_kt_k(a))\ge \frac{1}{2}P(|\sum_{k \ge 1}\alpha_k t_k(X) | \ge (1-\delta)||\{\alpha_k\sigma_k\}||_2).
\end{align}
Now
\begin{align*}
\E((\sum_{k \ge 1}\alpha_kt_k(X))^2)= \sum_{k \ge 1} \alpha_k^2\sigma_k^2,
\end{align*}
and hence (\ref{2.25}) and the Paley-Zygmund inequality implies 
\begin{align*}
P(\sum_{k\ge 1}\alpha_k t_k(X) &\ge \sum_{k\ge 1}\alpha_k t_k(a))\notag\\
\ge \frac{1}{2}P(|\sum_{k \ge 1}\alpha_kt_k(X) |^2 &\ge (1-\delta)^2||\{\alpha_k\sigma_k\}||^2_2)
\end{align*}
\begin{align}\label{2.29}
\ge \frac{1}{2}\{[1-(1-\delta)^2]\frac{\E((\sum_{k \ge 1}\alpha_kt_k(X))^2)}{(\E((\sum_{k \ge 1}\alpha_k  t_k(X))^4))^{\frac{1}{2}}}\}^2. 
\end{align}
Since (A-II) holds, Lemma \ref{lin-comb-4th-mom} implies we can combine (\ref{fourthpower}) and (\ref{2.29}) to obtain
\begin{align}
\inf_{\alpha \in \ell_0}P( \sum_{k \ge 1}\alpha_kt_k(X) \ge \sum_{k \ge 1}\alpha_kt_k(a)) \ge \frac{1}{6c} (2\delta-\delta^2)^2>0,
\end{align}
and the lemma is proven.
\end{proof}

\n{\bf Proof of Theorem \ref{nasc-indep} assuming (A-I)and (A-II).} First we observe from Theorem 1 that (\ref{series-for-a}) is necessary for (\ref{thm2hd>0}). Hence we turn to the converse. 

To prove sufficiency we first observe that for each $d \ge1$ and $\alpha \in \ell_0$
\begin{align}
P(\sum_{k \ge 1}\alpha_kt_k(X) \ge \sum_{k \ge 1}\alpha_kt_k(a))
\end{align}
\begin{align*}
\ge P(\sum_{k = 1}^d \alpha_kt_k(X) \ge \sum_{k = 1}^d\alpha_kt_k(a), \sum_{k \ge d+1} \alpha_kt_k(X) \ge \sum_{k \ge d+1}\alpha_kt_k(a)),
\end{align*}
and hence the independence of the $\{t_k(X):k \ge 1\}$ implies
\begin{align}
P(\sum_{k \ge 1}\alpha_kt_k(X) \ge \sum_{k \ge 1}\alpha_kt_k(a))
\end{align}
\begin{align*}
\ge P(\sum_{k = 1}^d \alpha_kt_k(X) \ge \sum_{k = 1}^d\alpha_kt_k(a))P( \sum_{k \ge d+1} \alpha_kt_k(X) \ge \sum_{k \ge d+1}\alpha_kt_k(a))).
\end{align*}
Taking $d$ sufficiently large such that 
\begin{align}
\sum_{k \ge d+1} \frac{t_k^2(a)}{ \sigma_k^2}<1,
\end{align}
we have from Lemma \ref{using-PZ} that
\begin{align}
\inf_{\alpha \in \ell_0}P( \sum_{k \ge d+1} \alpha_kt_k(X) \ge \sum_{k \ge d+1}\alpha_kt_k(a)))>0.
\end{align}
Since (\ref{A-I(i)}) holds we have
\begin{align}
\inf_{k \in \ell_0}P(\sum_{k = 1}^d \alpha_kt_k(X) \ge \sum_{k= 1}^d\alpha_kt_k(a))>0,
\end{align} 
and the theorem is proved when (A-I) and  (A-II) are assumed
\bigskip

 In order to complete the proof of the theorem under the assumptions in (A-III) we need a definition, some additional notation and some lemmas. We start with a definition of an admissible translate for a probability measure. Admissible translates appear in a variety of settings in the literature, sometimes with slightly different meanings, but for probability measures on the Borel subsets of $\mathbb{R}^{\infty}$ with i.i.d. coordinates our definition below agrees with that used by Shepp in \cite{shepp-admissible} for a totally indistinguishable translate. It also agrees with terminology used in the study of centered Gaussian measures $\mu$ on a separable Banach space, where the admissible translates are the vectors in the Hilbert space $H_{\mu}$ given in subsection 3.1, and in similar situations for other types of measures. 

\begin{mydef}\label{AdmTrans} Let $\mu$ be a probability measure on  the Borel subsets $\mathcal{B}_E$ of an F-space $E$, and for $x \in E$  and $A \in \mathcal{B}_E$ set
$\mu_x(A) = \mu(A-x).$ Then $x$ is an admissible translate of $\mu$ if $\mu_x$ and $\mu$ are mutually absolutely continuous with respect to one another on $(E,\mathcal{B}_E)$. 
\end{mydef}

\begin{lem}\label{-x adm}If $x$ is an admissible translate of the probability, $\mu$, then $-x$ is also an admissible translate of $\mu$.
\end{lem}

\begin{proof} Suppose $\mu(A+x)=0$. Then, $x$ an admissible translate of $\mu$ implies $\mu((A+x)-x)=0$, and hence $\mu(A)=0$. Conversely, if $\mu(A)=0$, then
$\mu((A+x)-x)=\mu(A)=0$, and since x is an admissible translate, this implies $\mu(A+x)=0$. 
\end{proof}
When $E$ is a sequence space, such as $\mathbb{R}^{\infty}$ or $\ell_{p}$, we denote the typical vector $x$ by writing $x=(x_1,x_2,\ldots)$ or $x=\{x_k: k \ge 1\}$.

\begin{lem}\label{move-R-infty} Let $X$ take values in the F-space $B$ and assume $\mu= \mathcal{L}(X)$ is  defined on $(B,\mathcal{B}_B)$. Let $\Lambda: B \rightarrow \mathbb{R}^{\infty}$ be such that
\begin{align}
\Lambda(a) = (t_1(a),t_2(a),\ldots), a \in B,
\end{align}
where the maps $t_k(\cdot), k \ge1,$ are $\mathcal{B}_B$ measurable to the reals, and  $\mathcal{L}(\Lambda(X))=\mathcal{L}(-\Lambda(X))$, i.e., $\Lambda(X)$ has a symmetric distribution. If $a \in B$ is such that $(t_1(a),t_2(a),\ldots)$ is an admissible translate for the probability measure $\mu^{\Lambda}(A)= \mu(\Lambda^{-1}(A))=P( X \in \Lambda^{-1}(A)), A \in \mathcal{B}_{\mathbb{R}^{\infty}}$, then 
\begin{align}\label{2.30}
HD_{\mathcal{T}}(a,\mu)>0, 
\end{align}
where $\mathcal{T}$ denotes all finite linear combination of the maps $\{t_k: k \ge 1\}$.
\end{lem}

\begin{proof}  If $HD_{\mathcal{T}}(a,\mu) = 0$, then there exists
$f_n \in \ell_0$ such that
\begin{align}
\lim_{n \rightarrow \infty} P(t_{ f_n}(X) \ge t_{f_n}(a)) =0.
\end{align}
By taking a subsequence, we may assume that
\begin{align}
\sum_{ n \ge 1}P( t_{f_n}(X) \ge t_{f_n}(a)) < \infty,
\end{align}
and therefore 
$P(t_{f_n}(X) \ge t_{f_n}(a)~ {\rm  {i.o.}}) = 0$. 

We now connect this with $\Lambda$ and the fact that $\Lambda(a)$ is an admissible translate. For this purpose for $n\ge 1$ let $f_n=(\alpha_{1,n},\alpha_{2,n},\ldots) \in \ell_0$,  $x=(x_1,x_2,\cdots) \in \mathbb{R}^{\infty}$, and  $ \langle f_n,x\rangle =\sum_{k \ge 1}\alpha_{k,n}x_k,$. We thus have 

\begin{align}\label{UseBorCant}
&P(t_{f_n}(X)- t_{f_n}(a)\ge 0~ {\rm i.o.})= P(\langle f_n, \Lambda(X)- \Lambda(a) \rangle \ge 0~{\rm{i.o.}})\\
& =\mu(u\in B \colon \langle f_n, \Lambda(u)- \Lambda(a) \rangle \ge 0~{\rm{i.o.}})=\mu^{\Lambda}(\alpha\in \mathbb{R}^{\infty} \colon \langle f_n, \alpha- \Lambda(a) \rangle \ge 0~{\rm{i.o.}})\notag\\
&=\mu^{\Lambda}(\beta+\Lambda(a)\in \mathbb{R}^{\infty} \colon \langle f_n,  \beta \rangle \ge 0~{\rm{i.o.}})=(\mu^{\Lambda})_{-\Lambda(a)}(\gamma\in \mathbb{R}^{\infty} \colon \langle f_n, \gamma \rangle \ge 0~{\rm{i.o.}})=0.\notag
\end{align}
Further,
\begin{align*}&(\mu^{\Lambda})_{-\Lambda(a)}(\gamma\in \mathbb{R}^{\infty} \colon \langle f_n, \gamma \rangle \ge 0~{\rm{i.o.}})=0\\
&\text{if and only if } (\mu^{\Lambda})_{-\Lambda(a)}(\gamma\in \mathbb{R}^{\infty} \colon \langle f_n, \gamma \rangle < 0,\ \text{eventually})=1.
\end{align*}

\n But, since $\Lambda(a)=(t_1(a),t_2(a),\cdots)$ is an admissible translate for the probability $\mu^{\Lambda}$ by Lemma \ref{-x adm} we also have $-\Lambda(a)$ is an admissible translate and consequently 
\begin{align}\label{UseAdmTransl}
1=\mu^{\Lambda}(\gamma\in \mathbb{R}^{\infty} \colon \langle f_n, \gamma \rangle < 0,\ \text{eventually})
\end{align}
By symmetry of the measure $\mu^{\Lambda}$ we also have 
\begin{align}&1=\mu^{\Lambda}(\gamma\in \mathbb{R}^{\infty} \colon \langle f_n, \gamma \rangle > 0,\ \text{eventually}),
\end{align}
which yields a contradiction.
\end{proof}

\n{\bf Proof of Theorem \ref{nasc-indep} assuming (A-III).} As before, Theorem \ref{necessity-zero-depth} shows (\ref{series-for-a}) is necessary for (\ref{thm2hd>0}). Hence we turn to the converse, showing 
\begin{align}\label{2.33}
\sum_{k\ge 1} \frac{t_k^2(a)}{\lambda_k^2} < \infty 
\end{align}
implies
\begin{align}\label{2.34}
HD_{\mathcal{T}}(a,\mu) >0. 
\end{align}

Since (\ref{phi-var}) holds, Lemma \ref{move-R-infty} will  show (\ref{2.34}) for $a \in B$ satisfying (\ref{2.33}) provided we show $\Lambda(a)=(t_1(a),t_2(a),\ldots)$ is an  admissible translate of $\mu^{\Lambda}$, where $\mu=\mathcal{L}(X).$

This follows using Kakutani's  result on the equivalence of infinite product measures as in \cite{shepp-admissible}. That is,  if $\mu_{k} = \mathcal{L}(t_k(X))$ and  $\nu_{k} = \mathcal{L}(t_k(X)+t_k(a))$ are mutually absolutely continuous for $k\ge1$, then
$\Lambda(a)$ is an admissible translate for $\mu^{\Lambda}$ if and only if
\begin{align}\label{2.35}
H(\mu^{\Lambda}, \mu^{\Lambda+\Lambda(a)})= \prod_{k=1}^{\infty} H(\mu_{k},\nu_{k})>0, 
\end{align}
where $\mu^{\Lambda+\Lambda(a)}= \mathcal{L}(\Lambda(X) + \Lambda(a))$ and
\begin{align}
H(\mu_{k},\nu_{k})= \int_{\mathbb{R}} (\frac{d\mu_{k}}{dx}\frac{d\nu_{k}}{dx})^{\frac{1}{2}}dx.
\end{align}
Now 
\begin{align}
\frac{d\mu_{k}}{dx}(s)=\frac{\phi(\frac{s}{\lambda_k})}{\lambda_k},
\end{align}
\begin{align}
\frac{d\nu_{k}}{dx}(s)=\frac{\phi(\frac{s-t_k(a)}{\lambda_k})}{\lambda_k},
\end{align}
and hence
\begin{align}\label{2.36}
H(\mu_{k},\nu_{k})= \int_{\mathbb{R}} (\frac{\phi(\frac{s}{\lambda_k})}{\lambda_k}\frac{\phi(\frac{s-t_k(a)}{\lambda_k})}{\lambda_k})^{\frac{1}{2}}ds=\int_{\mathbb{R}} (\phi(t)\phi(t-\frac{t_k(a)}{\lambda_k}))^{\frac{1}{2}}dt .
\end{align}
Now (\ref{2.33}) and $\phi$ having finite information, since (\ref{FisherInformation}) holds,  combine with part (ii) of Theorem 1 of \cite{shepp-admissible} to imply $\Lambda(a)/\lambda \equiv(\frac{t_1(a)}{\lambda_1},\frac{t_2(a)}{\lambda_2},\ldots)$ is an admissible translate of $P^{Y}$, where $Y= \Lambda(X)/\lambda$. Therefore, Kakutani's theorem implies
\begin{align}\label{2.37}
H(P^{Y},P^{Y+\Lambda(a)/\lambda} )>0, 
\end{align}
and since an easy calculation shows
\begin{align}\label{2.38}
H(P^{Y},P^{Y+\Lambda(a)/\lambda} )=\prod_{k=1}^{\infty}\int_{\mathbb{R}} (\phi(t)\phi(t-\frac{t_k(a)}{\lambda_k}))^{\frac{1}{2}}dt, 
\end{align}
(\ref{2.36}), (\ref{2.37}) and (\ref{2.38}) combine to imply (\ref{2.35}). Thus Kakutani's theorem implies $\Lambda(a)$ is an admissible translate of $\mu^{\Lambda}$ and Lemma \ref{move-R-infty} completes the proof.
\bigskip
\bigskip

\subsection{Proof of Theorem \ref{emp-depth}}
Since $\{t_k: k \geq 1\} \subseteq \mathcal{T}$ we have
\begin{align}\label{2.39}
HD_{\mathcal{T}}(a,\mu_n) \le \inf_{k \ge 1} \frac{1}{n} \sum_{j=1}^n I(t_k(X_j) \ge t_k(a)),  
\end{align}
and hence it suffices to show for all $n \ge 1$ and $a\in B$ satisfying (\ref{thm3ReqCond})
\begin{align}\label{2.40}
\inf_{k \ge 1} Z_{n,k}(a) = 0, 
\end{align}
with probability one, where
\begin{align}\label{2.41}
Z_{n,k}(a)=  \frac{1}{n} \sum_{j=1}^n I(t_k(X_j) \ge t_k(a)).  
\end{align}
Since the random variables $\{t_k(X): k \ge 1\}$ are independent, the sequences $\{t_k(X_j): k \ge 1\}$ consist of independent random variables, and as sequences are independent and identically distributed for $j \ge 1$. Hence fix $n \ge 1$ and assume $a \in B$ satisfies (\ref{thm3ReqCond}). 

Then, the sequence $\{Z_{n,k}(a): k \ge 1\}$ consists of independent random variables. Furthermore, (\ref{thm3-assump}) then
implies there exists $\delta>0$ and $\{k_i: i \ge 1\}$ a subsequence of the positive integers such that  for all $k_i$ 
\begin{align}
P(t_{k_i}(X)< t_{k_i}(a)) \ge \delta.
\end{align}
Therefore, for all $n \ge 1$
\begin{align}
P(Z_{n,k_i}(a) =0) \ge \delta^n.
\end{align}
Hence for $n$ fixed, the independence in $k\ge 1$ and the Borel-Cantelli lemma implies
\begin{align}\label{2.42}
P(Z_{n,k_i}(a)=0 ~{\rm { i.o.~in~i}})=1.
\end{align}
Now (\ref{2.42}) implies (\ref{2.40}) with probability one, and the theorem is proved.

\section{ Examples and explicit forms of half space depth} \label{examples}

If $\{t_k(X): k \ge 1 \}$ is a sequence of mean zero Gaussian random variables with variances $ \sigma_k^2 \in (0,\infty), k \ge 1$, then Theorems \ref{necessity-zero-depth} and \ref{nasc-indep}  readily apply to provide necessary and sufficient conditions for the half space depth in (\ref{thm1-conclusion})  to be positive. Theorem \ref{emp-depth} and Corollary \ref{consist} also provide information about the empirical depth. However, in this situation we can obtain an explicit formula for this depth. In fact, the formula we obtain holds when the sequence  $\{t_k(X): k \ge 1 \}$ consists of i.i.d. scaled symmetric stable random variables, which is interesting since among the non-degenerate stable random variables only the Gaussians have a variance and our conditions are in terms of second and higher moments. We also obtain a formula for the Tukey half space depth for any centered Gaussian measure on a separable Banach space $B$ when the depth is computed using $B^*$.
Finally, we provide some information when the $\{t_k(X): k \ge 1 \}$ are i.i.d. Rademacher random variables, but in this case the results are less explicit. 

\bigskip

\subsection{An explicit formula for half region depth for stables.} To obtain an explicit formula for the half space depth for symmetric stable random variables we need the following lemma on the sequence spaces $\ell_p, 0<p\le 2$. It is essentially proved in \cite{wheeden-zygmund}, pp. 128-129, and hence we omit further details.

\begin{lem} If $1<p<\infty, \frac{1}{p} + \frac{1}{q}=1,$ and $y=\{y_j: j\ge 1\} \in \mathbb{R}^{\infty}$, then
\begin{align}\label{3.1}
\sup_{x \in \ell_0,||x||_p =1}  \sum_{j \ge 1} |x_jy_j| =||y||_q, 
\end{align}
where $||y||_q=( \sum_{j \ge 1} |y_j|^q)^{\frac{1}{q}}$ could well be infinity. 
If $0<p\le 1,$ then (\ref{3.1}) holds with $q=\infty$ and $||y||_{\infty}= \sup_{k \ge 1}|y_k|$, which again could be infinite.
\end{lem}

\begin{notat} If $b=\{b_k: k \ge 1\} \in \mathbb{R}^{\infty}$ and $c=\{c_k: k \ge 1\}$ is a strictly positive sequence, we will write $b/c$ to denote the sequence $\{b_k/c_k: k \ge 1\}$, 
\end{notat}

\begin{prop} Let $S$ be a non-degenerate symmetric $p$-stable random variable where $0<p\le 2,$ and for $c_k \in (0,\infty), k\ge 1,$ assume $\{t_k(X): k \ge 1\}$ are independent with $\mathcal{L}(t_k(X))= \mathcal{L}(c_kS).$ If $\mathcal{T}$ denotes all finite linear combinations of $\{t_k: k \ge 1\}$ and for $a \in B$ we let $\tau(a)=\{t_k(a): k \ge 1\} $,
then 
\begin{align}\label{3.2}
HD_{\mathcal{T}}(a,\mu)= 1-P(S \le ||\tau(a)/c||_q), 
\end{align}
where $q=\infty$ for $0<p\le 1$, $\frac{1}{p} +\frac{1}{q}=1$ for $1<p\le 2$. 
\end{prop}

\begin{rem} If  $||\tau(a)/c||_q=\infty$ in (\ref{3.2}), then $HD_{\mathcal{T}}(a,\mu)=0$ for $0<p \le 2$. Moreover, since $c$ is fixed, the depth is continuous as a function of the sequence $\tau(a)/c$ in the $q$-norm when restricted to the set where $||\tau(a)/c|||_q<\infty$, but it is highly discontinuous with respect to the product topology on $\mathbb{R}^{\infty}.$  If $p=2$, $S$ has variance one, and $c_k=1, k \ge 1$, then for any $a \in B$
\begin{align}\label{3.3}
HD_{\mathcal{T}}(a,\mu)= 1-\Phi(||\tau(a)||_2), 
\end{align}
where $\Phi(\cdot)$ is the distribution function of a centered Gaussian random variable with variance one. Also, if $1<p\le 2$, then it is easy to see from Remark 4 that consistency fails at all $a \in B$ where the depth is strictly positive.
\end{rem}

\begin{proof} Since the sequence $\{t_k(X):k \ge 1\}$ is independent and symmetric, it  suffices to show 
\begin{align}
 \inf_{\alpha \in \ell_0, t_{ \alpha}(a) >0} P(\frac{t_{\alpha}(X)}{t_{\alpha}(a)} \ge 1)=1-P(S \le ||\tau(a)/c||_q) 
\end{align}
Now 
\begin{align}
\frac{t_{\alpha}(X)}{t_{\alpha}(a)}= \frac{\sum_{k\ge 1}\alpha_kt_k(X)}{\sum_{k \ge 1} \alpha_kt_k(a)},
\end{align}
and using independence and that the random variables $\{t_k(X), k \ge 1\}$ are $p$-stable, we therefore  have
\begin{align}
\mathcal{L}(\frac{t_{\alpha}(X)}{t_{\alpha}(a)})= \mathcal{L}(\frac{ (\sum_{k \ge 1}|\alpha_k c_k|^p)^{\frac{1}{p}} }{\sum_{k\ge 1}\alpha_kt_k(a)  } S).
\end{align}
Hence,
\begin{align}
 \inf_{\alpha \in \ell_0,t_{ \alpha}(a) >0} P(\frac{t_{\alpha}(X)}{t_{\alpha}(a)} \ge 1)=  \inf_{\alpha \in \ell_0, t_{\alpha}(a) >0} P(S \ge \frac{ \sum_{k\ge 1}\alpha_kt_k(a) }{(\sum_{k\ge 1}|\alpha_kc_k|^p)^{\frac{1}{p} }}).
\end{align}
Letting $\beta_k=|\alpha_k|{\rm{sign}}(t_k(a)), k \ge 1,$ and using the continuity of the distribution of $S$, we have
\begin{align}
 \inf_{\alpha \in \ell_0, t_{\alpha}(a) >0} P(\frac{t_{\alpha}(X)}{t_{\alpha}(a)} \ge 1)=P(S \ge  \sup_{\beta \in \ell_0,||\beta||_{\infty}>0}  \frac{ \sum_{k \ge 1}|\beta_kt_k(a)| }{(\sum_{k\ge 1}|\beta_k c_k|^p)^{\frac{1}{p} }}). 
 \end{align}
Setting 
\begin{align}
\gamma_k=\frac{\beta_kc_k}{(\sum_{k \ge 1}|\beta_k c_k|^p)^{\frac{1}{p} }}, k \ge 1,
\end{align}
we have
\begin{align}
 \sup_{\beta \in \ell_0,||\beta||_{\infty}>0}  \frac{ \sum_{k \ge 1}|\beta_kt_k(a)| }{(\sum_{k\ge 1}|\beta_k c_k|^p)^{\frac{1}{p} }}=  
 \sup_{\{\gamma \in \ell_0,||\gamma||_p=1\}}   \sum_{k \ge 1}|\gamma_kt_k(a)/c_k| = ||\tau(a)/c||_q, 
\end{align}
where the last equality follows from Lemma 4 provided $p$ and $q$ are related as indicated in the proposition.
\end{proof}

Our next result obtains the analogue of the p=q=2 case of Proposition 1 when $X$ is a centered Gaussian random vector $X$ taking values in a separable Banach space $B$ and 
the half space depth is given by 
\begin{align} \label{3.4}
HD_{B^*}(a,\mu)= \inf_{t \in B^{*}} P( t(X) \ge t(a)) 
\end{align}
where $B^*$ is the dual of $B$. As before, we assume $X$ is defined on the probability space $(\Omega, \mathcal{F},P)$ and is measurable from $\cl F$ to the Borel subsets of $B$, and let $\mu = {\cl L}(X)$. 

Let  $||\cdot||$ and $||\cdot||_{B^*}$ denote the norms on $B$ and $B^*$, respectively. Then, by the Fernique-Landau-Shepp result \cite{fernique-integrability}, \cite{landau-shepp} for some $s>0$
\begin{align}
\int_B\exp\{s||x||^2\}d\mu(x)<\infty,  
\end{align}
and hence the linear  map $S: B^*\rightarrow B$ given by the Bochner integral
\begin{align}
Sf= \int_Bxf(x)d\mu(x) 
\end{align}
is continuous from $B^*$ to $B$. The Hilbert space $H_{\mu}$ generating  $\mu$ is given by the completion of the range of $S$ with respect to the norm $||\cdot||_{\mu}$ obtained from the inner product 
\begin{align}
\langle Sf,Sg \rangle_{\mu}= \int_Bf(x)g(x)d\mu(x).
\end{align}
Moreover, $H_{\mu}$ can be viewed as a subset of $B$, since for $x \in H_{\mu}$
\begin{align}
\|x\| \leq \sigma(\mu)\|x\|_{\mu} 
\end{align}
where
\begin{align}\label{3.5}
\sigma(\mu) \equiv \sup_{\|f\|_{B^*}\leq 1} \left(\int_B f^2(x)d\mu(x)\right)^{1/2} < \infty.
\end{align}
It is also well known that the support of the Gaussian measure $\mu$ is given by the closure of $H_{\mu}$ in $B$, which we denote by $\bar H_{\mu}$, and that $\mu(H_{\mu})=0$ when $H_{\mu}$ is infinite dimensional. Of course, $H_{\mu}=\bar H_{\mu}$ when $H_{\mu}$ is finite dimensional, so in that case $\mu(H_{\mu})=1.$ 

Additional properties relating $H_{\mu}, \langle \cdot, \cdot\rangle_{\mu},B,$ and the measure $\mu$ can be found in Lemma 2.1 of \cite{view}. However, the above suffice to state our proposition on the half space depth in (\ref{3.4}) for Gaussian measures, and for the following useful lemma used in its proof. The proof of the lemma is in \cite{view}, where it appears in a slightly more general form as Lemma 2.2. Once we have this lemma, the remainder of the proof follows as in Proposition 1.
\bigskip

\begin{lem} Let $\mu$ be a centered Gaussian measure on $B,$ and assume $H_{\mu}$ and $||\cdot||_{\mu}$ are defined as above. If
\begin{align}\label{3.6}
\theta(x)= \sup_{f \in B^*,\int_Bf^2d\mu \le 1} f(x), 
\end{align}
then $\theta(x)=||x||_{\mu}$ for $x \in H_{\mu}$, and $\theta(x)= \infty$ for $x \in B-H_{\mu}$.
\end{lem}

\begin{prop}\label{gaussian} Let $X$ be a B-valued centered Gaussian random vector as above, and assume $\Phi(\cdot)$ is the distribution function of a mean zero-variance one Gaussian random variable. Then,
\begin{align}\label{3.7}
HD_{B^*}(a,\mu)= 1-\Phi(||a||_{\mu}), a \in H_{\mu}, 
\end{align}
and
\begin{align}\label{3.8}
HD_{B^*}(a,\mu)=0, a \in B-H_{\mu}. 
\end{align}
Furthermore,  
\begin{align}\label{3.9}
\mu(a \in B: HD_{B^*}(a,\mu)=0)= 0~ {\rm{or}}~ 1 
\end{align} 
according as  $H_{\mu}$ is a finite dimensional Hilbert space or an infinite dimensional Hilbert space. In addition, if $H_{\mu}$ is infinite dimensional, then for all $a \in H_{\mu}$ the empirical half space depth given in (\ref{ehsd}) or (\ref{ehsd'}) with $\mathcal{T}=B^{*}$ is zero with $P$-probability one. Hence the empirical half depth fails to approximate the true half space depth for all $a \in H_{\mu}$ in this setting, i.e. consistency fails at all such points. 
\end{prop}

\begin{proof} As mentioned following (\ref{3.5}), since $\mu$ is a Gaussian measure with mean vector zero, $\mu(H_{\mu})$ has probability one or zero, according as  $H_{\mu}$ is a finite dimensional Hilbert space or an infinite dimensional Hilbert space, so (\ref{3.9}) follows immediately once we verify (\ref{3.7}) and (\ref{3.8}).

Since $f(x) $ is centered Gaussian with variance $ \sigma_f^2\equiv  \int_Bf^2(x)d\mu(x)$ for $f \in B^*$, it follows that
\begin{align}
HD_{B^*}(a,\mu)= \inf_{f \in B^*, f(a)>0} \mu(x:f(x) \ge f(a))= \inf_{f \in B^*, f(a)>0}[1- \Phi(f(a)/\sigma_f)].
\end{align}
Using the continuity of $\Phi$ and that it is increasing, we thus have
\begin{align}
HD_{B^*}(a,\mu)= 1- \Phi( \sup_{f \in B^*, f(a)>0} f(a)/\sigma_f).
\end{align}
Moreover, since
\begin{align}
 \sup_{f \in B^*, f(a)>0}f(a)/\sigma_f
 =\theta(a), 
 \end{align}
where $\theta(\cdot)$ is as in (\ref{3.6}), we therefore have (\ref{3.7}) and (\ref{3.8}).

If $H_{\mu}$ is infinite dimensional, then there exists a sequence $\{t_k: k \ge 1\} \subseteq B^{*}$ such that $\{St_k=S(t_k): k \ge 1\}$ are orthonormal in $H_{\mu}$,
for all $a \in H_{\mu}$,
$$
\lim_{k \rightarrow \infty} t_k(a)=0,
$$
and $\{t_k(X): k \ge 1\}$ are independent centered Gaussian random variables with variance one. Thus, for all $a \in H_{\mu}$ and $\mu_n= \frac{1}{n}\sum_{j=1}^n\delta_{X_j}$, where $X,X_1,X_2,\cdots$ are i.i.d. $B$-valued Gaussian random vectors, the empirical depth
$$
HD_{B^{*}}(a, \mu_n)= \inf_{ t \in B*} \mu_{n}(x \in B: t(x) \ge t(a))\le \inf_{k \ge 1}\frac{1}{n} \sum_{j=1}^n I(t_k(X_j) \ge t_k(a)).  
$$
Hence, as in the proof of Theorem \ref{emp-depth} for $n \ge 1$ fixed, the independence in $k \ge 1$ and the Borel-Cantelli lemma easily imply the empirical half space depth is zero with $P$-probability one.
Thus, the proposition is proved.
\end{proof}

\subsection{The Rademacher case.}\label{rad-case} The explicit results obtained in Propositions 1 and 2 depend on the scaling properties of the symmetric stable laws, and therefore are likely quite special. They also involve continuous distributions, so for contrast we examine the special discrete case where $\{t_k(X):k\ge 1\}$ are  independent Rademacher  random variables. Of course, Theorem 1 implies that 
\begin{align}\label{3.10}
HD_{\mathcal{T}}(a,\mu)=0 
\end{align}
whenever  $\sum_{k \ge 1} t_k^2(a) = \infty$, but, as can be seen from Lemma 7 below, that is not the entire story. Furthermore, although the condition (\ref{A-I(i)}) in (A-I) of Theorem \ref{nasc-indep} is not applicable, once we prove Lemma 7, it is easy to see how a suitable modification of the proof of Theorem \ref{nasc-indep} and the conditions in (A-I) allow us to identify the set where the depth is strictly positive. We also indicate how the ideas in Montgomery-Smith's paper \cite{monty-rad-sums} provides an alternative approach to obtain the results for Rademacher $\{t_k(X): k\ge 1\}$. Finally, we point out that Lemma 7 can be refined to apply to other bounded random variables, and hence Proposition 3 below has comparable analogues. For example, if $\{t_k(X): k \ge 1\}$ are independent, symmetric random variables with $\E(t_k^2(X)) \in (0,\infty)$, (A-II) holds, $P(|t_k(X)| \le b_k)=1$ where $b_k<\infty$ for all $k \ge 1$,
and $\mathcal{L}(t_k(X))$ has support in every neighborhood of $b_k$ and of $-b_k$ for all $k \ge 1$, then
\begin{align*}
\sup_{k \ge 1} \frac{|t_k(a)|}{b_k}>1~{\rm {or}}~\sum_{k \ge 1}\frac{t_k^2(a)}{\sigma_k^2} = \infty ~{\rm{imply}}~HD_{\mathcal{T}}(a,\mu)=0, 
\end{align*}
and
\begin{align*}
\sup_{k \ge 1} \frac{|t_k(a)|}{b_k}<1~{\rm {and}}~\sum_{k \ge 1}\frac{t_k^2(a)}{\sigma_k^2} < \infty~{\rm{imply}}~HD_{\mathcal{T}}(a,\mu)>0. 
\end{align*}
What happens when $\sup_{k \ge 1} \frac{|t_k(a)|}{b_k}=1$ depends on whether $\mathcal{L}((t_k(X))$ has positive mass at $b_k$, or not. 
This can be seen in the  following proposition, which  summarizes our results for Rademacher variables. Since its proof can easily be modified to obtain the previous results, those details are omitted.
\bigskip

\begin{prop} Let $\{t_k(X):k\ge 1\}$ be independent Rademacher random variables. Then,  (\ref{3.10})  holds for all $a \in B$ such that  $\sum_{k \ge 1} t_k^2(a) = \infty$  or  $\sup_{k \ge 1}|t_k(a)|>1.$ In addition, 
\begin{align}\label{3.11}
HD_{\mathcal{T}}(a,\mu)>0 
\end{align}
for all $a \in B$ such that  $\sum_{k \ge 1} t_k^2(a) < \infty$  and  $\sup_{k \ge 1}|t_k(a)|\le 1,$ and consistency fails at all such points $a \in B$.
\end{prop}

The proof of the proposition requires the following lemma.
\bigskip

\begin{lem} Let $\{t_k(X):k \ge1\}$ be independent Rademacher random variables. If  $a\in B$ and  $\sup_{k \ge 1}|t_k(a)|>1$, then 
\begin{align}\label{3.12}
HD_{\mathcal{T}}(a,\mu) =0 . 
\end{align}
Furthermore, if  $\sup_{k \ge 1}|t_k(a)|\le 1$, then for every $d\ge 1$ 
\begin{align}\label{3.13}
HD_{\mathbb{R}^d}(\Pi_d(a),\mu^{\Pi_d})\ge 2^{-d},  
\end{align}
where $\Pi_d(a)=(t_1(a),\ldots,t_d(a)), a \in B$, and $\mu^{\Pi_d}$ is the image of $\mu$ on $\mathbb{R}^d$ via the map $\Pi_d(\cdot): B \rightarrow \mathbb{R}^d.$
\end{lem}

\begin{proof} If $|t_{k_0}(a)|>1,$ and  $\widehat\alpha= \{\delta(k,k_0){\rm{sign}} (t_k(a)): k \ge 1\}$ where $\delta(k,k_0)=1$ when $k=k_0$ and zero otherwise, then
\begin{align}
HD_{\mathcal{T}}(a,\mu)= \inf_{\alpha \in \ell_0} P(\sum_{k\ge1}\alpha_k t_k(X)\ge \sum_{k \ge1}\alpha_kt_k(a))\le \inf_{\alpha= \widehat\alpha}P(\alpha(X) \ge \alpha(a))
\end{align}
\begin{align}
=P({\rm {sign}}(t_{k_0}(a))t_{k_0}(X) \ge |t_{k_0}(a)|)=0.
\end{align}
Hence, (\ref{3.12}) holds.

To verify (\ref{3.13}) we observe that $ \mu^{\Pi_d} = \mathcal{L}( P^{\Pi_d(X)})$, and hence
\begin{align*}
HD_{\mathbb{R}^d}(\Pi_d(a),\mu^{\Pi_d})=  \inf_{ \alpha \in \ell_0} P(\sum_{k=1}^d \alpha_kt_k(X) \ge \sum_{k=1}^d \alpha_kt_k(a)) 
\end{align*}
\begin{align*}
\ge \inf_{\alpha \in \ell_0} P(\sum_{k =1}^{d-1}\alpha_k t_k(X) \ge \sum_{k=1}^{d-1}\alpha_k t_k(a), \alpha_dt_d(X)\ge \alpha_dt_d(a)) 
\end{align*}
\begin{align}\label{3.14}
\ge \inf_{\alpha \in \ell_0} P(\sum_{k =1}^{d-1}\alpha_k t_k(X) \ge \sum_{k=1}^{d-1}\alpha_k t_k(a))P( \alpha_dt_d(X)\ge \alpha_dt_d(a)),  
\end{align}
where the last inequality holds by the independence of the $\{t_k(X): k \ge 1\}$. Furthermore, for all $k\ge1,$ 
\begin{align}\label{3.15}
\inf_{s \in \mathbb{R}}P(s t_k(X) \ge s t_k(a)) ={\rm {min}}[A_1,A_2], 
\end{align}
where
\begin{align*}
A_1=\inf_{s \ge 0}P(s t_k(X) \ge s t_k(a))=P(t_k(X) \ge t_k(a)) 
\end{align*}
and 
\begin{align*}
A_2=\inf_{s <0}P(s t_k(X) \ge s t_k(a))=P(t_k(X) \le t_k(a)).
\end{align*}
Hence $|t_k(a)| \le 1$ implies $A_i \ge 1/2$ for $i=1,2$, and (\ref{3.15}) then implies (\ref{3.13}) for $d=1$. Furthermore, (\ref{3.15}) applied to (\ref{3.14}) allows us to induct on 
$d$ proving (\ref{3.13}) for all $d \ge 1$.
\end{proof}

\n{\bf Proof of Proposition 3.} If $\sum_{k \ge 1} t_k^2(a) = \infty$, then (\ref{3.10}) holds by Theorem \ref{necessity-zero-depth}, and  when $\sup_{k \ge 1}|t_k(a)| > 1,$ we have (\ref{3.10}) by Lemma 7. Hence it remains to show that $\sum_{k \ge 1} t_k^2(a) < \infty$  and  $\sup_{k \ge 1}|t_k(a)|\le 1$ imply (\ref{3.11}).
This follows since (\ref{3.13}) holds for all $a\in B$ satisfying $\sup_{k \ge 1}|t_k(a)|\le 1,$ and hence, although this is not equivalent (\ref{A-I(i)}) in (A-I), the proof of Theorem \ref{nasc-indep} shows that if (\ref{A-I(i)}) is replaced by (\ref{3.13}) in (A-I), then (\ref{3.11}) holds provided $\sum_{k \ge 1} t_k^2(a) < \infty$  and  $\sup_{k \ge 1}|t_k(a)|\le 1$ for $a \in B$. Finally, Remark 4 following Theorem 3 implies that for all such $a \in B$ consistency fails.
\bigskip

Although Proposition 3 identifies those $a \in B$ with positive half-space depth for the Rademacher variables, it is unclear what its value might be on such points. Below we obtain some estimates on a lower bound for this depth using two different methods. The first method modifies the estimates in Lemma \ref{using-PZ} appropriately, and the second applies the delicate results in \cite{monty-rad-sums}. However, neither approach yields estimates that apply to all $a \in B$ where the half space depth is positive, and hence they do identify that collection of points as in Proposition 3.

\begin{prop} Let $\{t_k(X): k \ge 1\}$ be independent  Rademacher random variables, and $a \in B$. If $r\in\mathbb{N}$ and $\delta$ are such that 
 $ \sum_{k >r}t_k^2(a)\le \dfrac14$ and $\delta\sqrt{r}\le \dfrac14$, then  $\sup_{k \ge 1}|t_k(a)|\le\delta$ implies
 \begin{align}
HD_{\mathcal{T}}(a,\mu)=\inf_{\alpha \in \ell_0}P( \sum_{k \ge 1}\alpha_k t_k(X) \ge \sum_{k \ge 1}\alpha_k t_k(a))\ge \dfrac3{32}.
\end{align}
\end{prop}

\begin{proof} If $r,\delta>0$ and $\alpha \in \ell_0$ are as indicated in the proposition, then 
\begin{align*}&\sum_{k \ge 1}\alpha_{k}t_k(a) = \sum_{k=1}^{r}\alpha_{k}t_k(a)+\sum_{ k\ge r+1}\alpha_{k}t_k(a)\le \delta\sqrt{r}\|\alpha\|_{2}+\|\alpha\|_{2}\bigl(\sum_{k>r}t_k^2(a)\bigr)^{1/2}.
\end{align*}
Hence, $\sum_{k \ge 1}\alpha_{k}t_k(a) \le \dfrac12\|\alpha\|_{2}$ , and
\begin{align}
&P\bigl(\sum_{k\ge 1}\alpha_kt_k(X) \ge\sum_{k \ge 1}\alpha_{k}t_k(a)\bigr)\ge P\bigl(\sum_{k \ge 1}\alpha_kt_k(X)\ge\dfrac12 \|\alpha\|_{2}\bigr)
\end{align}
\begin{align}\label{prePZ}&=\dfrac12 P\bigl(|\sum_{k\ge 1}\alpha_{k}t_k(X)|\ge\dfrac12 \|\alpha\|_{2}\bigr)
=\dfrac12P\bigl(|\sum_{k \ge 1}\alpha_kt_k(X)|^{2}\ge\dfrac14 \|\alpha\|_{2}^{2}\bigr).
\end{align}
Thus the Paley-Zygmund inequality and Lemma 1 (with the $t_k(X)$ Rademacher random variables) applied to the last term in (\ref{prePZ}), imply
\begin{align}
P\bigl(\sum_{k\ge 1}\alpha_kt_k(X) \ge\sum_{k\ge 1} \alpha_kt_k(a)) \ge  \dfrac{9}{32} 
\dfrac{(E[ (\sum_{k\ge 1}\alpha_{k}t_k(X))^{2}])^2}{E[(\sum_{k \ge 1}\alpha_k t_k(X))^{4}]}\ge\dfrac{3}{32},
\end{align}
which proves the proposition.
\end{proof}

In order to use the results in \cite{monty-rad-sums} we introduce some norms from the theory of interpolation of Banach spaces. Of course,  (\ref{3.18}) below plays a role analogous to the $\ell_{\infty}$ and $\ell_2$ assumptions in Proposition 3. The notation is from  \cite{monty-rad-sums}, which defines for $x \in \ell_2$ and $t>0$
\begin{align}
K_{1,2}(x,t)\equiv K(x,t:\ell_1,\ell_2)= \inf\{||x'||_1 +t||x''||_2: x',x'' \in \ell_2, x'+x''=x\},
\end{align}
and
\begin{align}
J_{\infty,2}(x,t) \equiv J(x,t:\ell_{\infty},\ell_2)= \max \{||x||_{\infty}, t||x||_2\}.
\end{align}
Then, for $t>0, x \in \ell_2$ we have from Lemma 1 in  \cite{monty-rad-sums} that
\begin{align}\label{3.16}
K_{1,2}(x,t)= \sup\{\sum_{k\ge 1}x_ky_k: y \in \ell_2, J_{\infty,2}(y,t^{-1}) \le 1\}, 
\end{align}
and Theorem 1 of  \cite{monty-rad-sums} implies there is a constant $c>0$ such that for all $x \in \ell_2$ and $t>0$
\begin{align}\label{3.17}
P(\sum_{k \ge 1}x_k\epsilon_k \ge c^{-1}K_{1,2}(x,t)) \ge c^{-1}e^{-ct^2}, 
\end{align}
where $\{ \epsilon_k:k \ge 1\}$ are independent Rademacher random variables.
\bigskip

\begin{prop} Let $\{t_k(X): k \ge 1\}$ be independent  Rademacher random variables, and $a \in B$ is such that $\sum_{k \ge 1}t_k^2(a)< \infty$ . If $c>0$ is as in (\ref{3.17}) and for some $t_0>0$
\begin{align}\label{3.18}
\max\{c||\{t_k(a):k \ge 1\}||_{\infty}, t_0 ^{-1}c||\{t_k(a):k \ge 1\}||_2\} \le 1, 
\end{align}
then,
\begin{align}\label{3.19}
HD_{\mathcal{T}}(a,\mu) \ge c^{-1}e^{-ct_0^2}. 
\end{align}
\end{prop}

\begin{proof} If (\ref{3.18}) holds, then (\ref{3.16}) implies for all $\alpha \in \ell_2$ that
\begin{align}\label{3.20}
K_{1,2}(\alpha,t_0) \ge c\sum_{k\ge 1} \alpha_kt_k(a). 
\end{align}
Since symmetry implies
\begin{align}\label{3.21}
HD_{\mathcal{T}}(a,\mu) = \inf_{\alpha \in \ell_0, t_{\alpha}(a)>0}P(\sum_{k\ge 1}\alpha_kt_k(X) \ge \sum_{k\ge1}\alpha_kt_k(a)), 
\end{align}
(\ref{3.17}),(\ref{3.20}), and (\ref{3.21}) combine to imply
\begin{align}
HD_{\mathcal{T}}(a,\mu) \ge \inf_{\alpha \in \ell_0, t_{\alpha}(a)>0}P(\sum_{k\ge 1}\alpha_k t_k(X) \ge c^{-1}K_{1,2}(\alpha,t_0))\ge c^{-1}e^{-ct_0^2}.
\end{align}
Therefore, (\ref{3.19}) holds and the proposition is proved.
\end{proof}
\bigskip

\subsection{Empirical depths for the Mosler-Polyakova version of Liu's simplicial depth.}\label{sec:MP} The depths considered to this point have been based on linear combinations of the one dimensional functionals $\{t_k:k \ge 1\}$, but they may as well take values in $\mathbb{R}^d$. The recent manuscript by \cite{mosler-poly-funct-depth} uses this approach to define depths on $B$. That is, let $D_d(\cdot,\cdot)$ be a depth on $\mathbb{R}^d$, and assume $\Theta$ is a collection of Borel measurable maps from $B$ to $\mathbb{R}^d$. Then, for $\mu$ a Borel probability measure on $B$ and $a \in B$, define
\begin{align}\label{3.31}
D_{\Theta}(a,\mu)= \inf_{\theta \in \Theta} D_d(\theta(a), \mu^{\theta}), 
\end{align}
where $\mu^{\theta}(A)= \mu(\theta^{-1}(A)),$ A a Borel subset of $\mathbb{R}^d$. 

In connection with their application to data clouds in $B$, the paper \cite{mosler-poly-funct-depth} points out that there may be problems with this sort of depth when $\Theta$ is too large. The next proposition provides an explicit  example of this problem in connection with the empirical estimation of $D_{\Theta}(a,\mu)$ when $B=\mathbb{R}^{\infty}$ and $D_d(\cdot,\cdot)$ is simplicial depth as in \cite{liu}. That is, for $ x \in \mathbb{R}^d$ and $Q$ a Borel probability measure on $\mathbb{R}^d$
\begin{align}\label{3.32}
D_d(x,Q) = P( x\in  co(Y_1,\ldots,Y_{(d+1)})), 
\end{align}
where $Y_1,\ldots,Y_{(d+1)}$ are i.i.d. with law $Q$ and $co(Y_1,\ldots,Y_{(d+1)})$ denotes the open convex hull of $Y_1,\ldots,Y_{(d+1)}$.
In particular, it is interesting to observe via (\ref{3.37})-(\ref{3.40}) below that empirical estimation is not dependable when enough independence is inherent in the data, even if $\Theta$ is only countable. 

Some further notation is as follows. Let $X, X_1,X_2,\ldots$ be i.i.d. $\mathbb{R}^{\infty}$ valued random vectors with $X=(\eta_1,\eta_2,\ldots)$ where $\eta,\eta_1,\eta_2,\ldots$ are i.i.d. real valued random variables, and $X_j=(\eta_{1,j},\eta_{2,j},\ldots), j \ge 1$. For $x=(x_1,x_2,\ldots) \in \mathbb{R}^{\infty}$, let
\begin{align}\label{3.33}
\theta_k(x)=(x_{i_k +1},\ldots,x_{i_{k+1}}), 
\end{align}
where $i_k=(k-1)d, k=1,2,\ldots, $ and henceforth assume $\Theta=\{\theta_k: k \ge 1\}$. Then, for $a=(a_1,a_2,\ldots,a_d,a_1,a_2,\ldots,a_d, \ldots) \in \mathbb{R}^{\infty}$ we have $\theta_k(a)=(a_1,\ldots,a_d)$ for all $k \ge 1.$ Furthermore, for all $k \ge 1, u \ge 0$, the probability
\begin{align}\label{3.34}
P( \theta_k(a) \in {\rm{co}}( \theta_k(X_{u+1}), \ldots, \theta_k(X_{u+(d+1)}))) 
\end{align}
is independent of $k$ and $u$, and denoted by $ \lambda(a)$.  
For $a \in \mathbb{R}^{\infty}$, 
we define
\begin{align}\label{3.35}
Z_{n,k}(a)=  \sum_{J_{n,d}} I(\theta_k(a) \in {\rm{co}}(\theta_k(X_{j_1}),\ldots,\theta_k(X_{j_{d+1}}))), 
\end{align}
where 
$$
J_{n,d}= \{(j_1,\ldots,j_{d+1}): 1 \le j_1<\cdots<j_{d+1} \le n\}.
$$
Then, for $a \in  \mathbb{R}^{\infty}$, and $D_d(x,Q)$ the simplicial depth of (\ref{3.32}) with $Q=\mathcal{L}(\eta_1,\ldots,\eta_d),$ we follow \cite{liu-simplicial-88} and \cite{liu}, and consider the sample analogue of (\ref{3.31}) to be
\begin{align}\label{3.36}
D_{\Theta,n}(a)= \inf_{k \ge 1} \frac{Z_{n,k}(a)}{N_{n,d}},  
\end{align}
where 
$N_{n,d}=\frac{n!}{(d+1)!(n-d-1)!}$. This is slightly different than what one would have if the empirical simplicial depth were defined in terms of the empirical probability  measure $P_n =\frac{1}{n}\sum_{k=1}^n \delta_{X_k}$, as then  
$N_{n,d}$ would be replaced by $\frac{n^{(d+1)}}{(d+1)!}$. However, since these quantities differ by a non-random quantity which is $O(\frac{1}{n})$, we lose no generality in using (\ref{3.36}).

\begin{prop} Let $X,X_1,\cdots$ be  i.i.d. $\mathbb{R}^{\infty}$-valued Borel measurable random vectors on the probability space $(\Omega,\mathcal{F},P)$ as indicated above, and assume $\eta$ has a probability density on $\mathbb{R}$. Also, assume
$D_{\Theta}(a,\mu)$ is defined using simplicial depth on $\mathbb{R}^d$ as above, and $\Theta=\{\theta_k: k \ge 1\}$. Then, $\{\theta_k(X): k \ge 1\}$ are i.i.d. $\mathbb{R}^d$-valued random vectors with absolutely continuous distribution $Q$ on $\mathbb{R}^d$, and for $a=(a_1,a_2,\ldots,a_d,a_1,a_2,\ldots,a_d, \ldots) \in \mathbb{R}^{\infty}$ we have
\begin{align}\label{3.37}
D_{\Theta}(a,\mu)= \inf_{k \ge 1}P(\theta_k(a) \in {\rm{co}}(\theta_k(X_1),\ldots,\theta_k(X_{d+1})))=\lambda(a).
\end{align}
Furthermore, for each $k \ge 1$ we have with $P$--probability one that
\begin{align}\label{3.38}
\lim_{n \rightarrow \infty}|\frac{Z_{n,k}(a)}{N_{n,d}} - \lambda(a)|=0, 
\end{align}
but 
with probability one 
\begin{align}\label{3.39}
|D_{\Theta,n}(a)-D_{\Theta}(a,\mu)|=|\inf_{k\ge 1}\frac{Z_{n,k}(a)}{N_{n,d}}- \lambda(a)| = \lambda(a). 
\end{align}
Hence, for  $\lambda(a)>0$, the empirical simplicial depth fails to approximate the true simplicial depth as $n \rightarrow \infty$ in this model.
\end{prop}
\bigskip

\begin{rem} For the proof of the proposition we only need that $\eta$ does not have an atom of size one, but when using the open convex hull in the definition of simplicial depth something close to absolute continuity is needed to have $\lambda(a)>0$ for a large collection of points. Furthermore, it is also important in the proofs of various useful  properties of the simplicial depth in $\mathbb{R}^d$ for $d\ge 2$. For example, see the results in \cite{liu}. Finally, if the integer $d\ge 2$ is fixed, then essentially the same proof provides an analogous result provided we take $\eta,\eta_1,\ldots$ to be $\mathbb{R}^d$ valued random vectors with absolutely continuous distribution.
\end{rem}
\bigskip

{\bf Proof.} Since $\theta_k(a)=(a_1,\cdots,a_d)$ for $ k \ge 1,$ the independence structure we have assumed in $n$ and $k$ implies that (\ref{3.38}) follows from the law of large numbers for $U$-statistics, see, for example, \cite{H}, or Theorem 4.1.4 in \cite{dePG}.
Moreover,
$$
\{Z_{n,k}(a)=0\}= \{ \sum_{J_{n,d}}I(\theta_k(a) \notin {\rm{co}}(\theta_k(X_{j_1}),\ldots,\theta_k(X_{j_{d+1}})))= N_{n,d}\},
$$
where
$$
I(\theta_k(a) \notin {\rm{co}}(\theta_k(X_{j_1}),\ldots,\theta_k(X_{j_{d+1}})))
$$
$$
=I(\theta_k(a) \notin {\rm{co}}((\eta_{i_{k} +1,j_1},\ldots,\eta_{i_{k}+1, j_{d+1}}),\ldots,(\eta_{i_{k+1},j_1},\ldots, \eta_{i_{k+1},j_{d+1}}))).
$$
Since $\theta_k(a)=(a_1,\cdots,a_d)$ we therefore have 
$$
\{Z_{n,k}(a)=0\} \supseteq A_{1,n,k} \cup A_{2,n,k},
$$
where
$$
A_{1,n,k}= \{ \sum_{J_{n,d}}I(a_1<\eta_{i_{k} +1,j_1},\ldots,a_1< \eta_{i_{k}+1, j_{d+1}})= N_{n,d}\},
$$
and 
$$
A_{2,n,k}= \{ \sum_{J_{n,d}}I(a_1>\eta_{i_{k} +1,j_1},\ldots,a_1> \eta_{i_{k}+1, j_{d+1}})= N_{n,d}\}.
$$
Now
$$
A_{1,n,k}= \{ a_1<\eta_{i_{k} +1,1},\ldots,a_1< \eta_{i_{k}+1,n}\},
$$
and 
$$
A_{2,n,k}= \{ a_1>\eta_{i_{k} +1,1},\ldots,a_1> \eta_{i_{k}+1,n}\},
$$
and hence
$$
P(A_{1,n,k})=P(\eta>a_1)^n { \rm~{and}}~P(A_{2,n,k})=P(\eta<a_1)^n.
$$
Since $\eta$ has a continuous distribution function $P(\eta<a_1)+P(\eta>a_1)=1$ and hence
\begin{align}\label{3.40}
P(Z_{n,k}(a)=0)\ge P(\eta>a_1)^n + P(\eta<a_1)^n>0, 
\end{align}
Applying the Borel-Cantelli lemma with $n \ge 1$ fixed, the independence in $k $ and (\ref{3.40})  implies
\begin{align}\label{3.41}
P(Z_{n,k}(a)=0 ~{\rm {i.o.~in~k}})=1.
\end{align}
and hence $P(\inf_{k \ge 1}Z_{n,k}(a)=0)=1$.
Thus (\ref{3.41}) implies (\ref{3.39}) with probability one, which proves the proposition.
\bigskip

\noindent {\bf Acknowledgement:} It is a pleasure to thank the referees for a careful reading of the manuscript. Their comments and suggestions led to a number of improvements in the exposition.

\providecommand{\bysame}{\leavevmode\hbox to3em{\hrulefill}\thinspace}
\providecommand{\MR}{\relax\ifhmode\unskip\space\fi MR }
\providecommand{\MRhref}[2]{%
  \href{http://www.ams.org/mathscinet-getitem?mr=#1}{#2}
}
\providecommand{\href}[2]{#2}


\begin{thebibliography}{dlPG99}

\bibitem[CC13]{chak-chaud-12}
Anirvan Chakraborty and Probal Chaudhuri, \emph{On data depth in infinite
  dimensional spaces}, accepted in the Annals of the Institute of Statistical Mathematics
  (2013).

\bibitem[CF09]{cuevas-fraiman-dual}
Antonio Cuevas and Ricardo Fraiman, \emph{On depth measures and dual
  statistics. {A} methodology for dealing with general data}, J. Multivariate
  Anal. \textbf{100} (2009), no.~4, 753--766. \MR{2478196 (2010c:62126)}

\bibitem[DGC11]{dutta-tukey}
Subhajit Dutta, Anil~K. Ghosh, and Probal Chaudhuri, \emph{Some intriguing
  properties of {T}ukey's half-space depth}, Bernoulli \textbf{17} (2011),
  no.~4, 1420--1434. \MR{2854779 (2012k:62147)}

\bibitem[dlPG99]{dePG}
V{\'{\i}}ctor~H. de~la Pe{\~n}a and Evarist Gin{\'e}, \emph{Decoupling},
  Probability and its Applications (New York), Springer-Verlag, New York, 1999,
  From dependence to independence, Randomly stopped processes. $U$-statistics
  and processes. Martingales and beyond. \MR{99k:60044}

\bibitem[Fer70]{fernique-integrability}
Xavier Fernique, \emph{Int\'egrabilit\'e des vecteurs gaussiens}, C. R. Acad.
  Sci. Paris S\'er. A-B \textbf{270} (1970), A1698--A1699. \MR{0266263 (42
  \#1170)}

\bibitem[Fis73]{Fisher73}
Ronald~A. Fisher, \emph{Statistical methods and scientific inference}, Hafner
  Press [A Division of Macmillan Publishing Co., Inc.], New York, 1973, Third
  edition, revised and enlarged. \MR{0346955 (49 \#11675)}

\bibitem[Hoe61]{H}
Wassily Hoeffding, \emph{The strong law of large numbers for u-statistics},
  Institute of Stat. Mimeo Ser. No. 302 (1961).

\bibitem[Kak48]{kakutani-prod-equiv}
Shizuo Kakutani, \emph{On equivalence of infinite product measures}, Ann. of
  Math. (2) \textbf{49} (1948), 214--224. \MR{0023331 (9,340e)}

\bibitem[KKZ13]{kkz}
James Kuelbs, Thomas Kurtz, and Joel Zinn, \emph{A clt for empirical processes
  involving time dependent data},  Ann. Probab. (2013), vol. 41, No. 2, 785-816.

\bibitem[KZ08]{view}
Jim Kuelbs and Joel Zinn, \emph{Another view of the clt in {B}anach spaces},
  Journal of Theoretical Probability \textbf{21} (2008), no.~4, 982--1029.

\bibitem[KZ13]{kz-quant}
James Kuelbs and Joel Zinn, \emph{Empirical quantile clt's for time dependent
  data}, Progress in Probab. HDP-VI (2013), no.~66, 169-196.

\bibitem[KZ12]{kz-half-region}
\bysame, \emph{Half-region depth for stochastic processes}, submitted, October 2012.

\bibitem[Liu88]{liu-simplicial-88}
Regina~Y. Liu, \emph{On a notion of simplicial depth}, Proc. Nat. Acad. Sci.
  U.S.A. \textbf{85} (1988), no.~6, 1732--1734. \MR{930658 (89c:62090)}

\bibitem[Liu90]{liu}
\bysame, \emph{On a notion of data depth based on random simplices}, Ann.
  Statist. \textbf{18} (1990), no.~1, 405--414. \MR{91d:62068}

\bibitem[LPR09]{lp-r-concept}
Sara L{\'o}pez-Pintado and Juan Romo, \emph{On the concept of depth for
  functional data}, J. Amer. Statist. Assoc. \textbf{104} (2009), no.~486,
  718--734.

\bibitem[LPR11]{lp-r-half}
\bysame, \emph{A half-region depth for functional data}, Comput. Statist. Data
  Anal. \textbf{55} (2011), no.~4, 1679--1695.

\bibitem[LS70]{landau-shepp}
H.~J. Landau and L.~A. Shepp, \emph{On the supremum of a {G}aussian process},
  Sankhy\=a Ser. A \textbf{32} (1970), 369--378. \MR{0286167 (44 \#3381)}

\bibitem[MP12]{mosler-poly-funct-depth}
Karl Mosler and Yulia Polyakova, \emph{General notions of depth for functional
  data}, arxiv preprint, 2012.

\bibitem[MS90]{monty-rad-sums}
S.~J. Montgomery-Smith, \emph{The distribution of {R}ademacher sums}, Proc.
  Amer. Math. Soc. \textbf{109} (1990), no.~2, 517--522. \MR{1013975
  (91a:60034)}

\bibitem[Rud66]{rudin}
Walter Rudin, \emph{Real and complex analysis}, McGraw-Hill Book Co., New York,
  1966. \MR{0210528 (35 \#1420)}

\bibitem[She65]{shepp-admissible}
L.~A. Shepp, \emph{Distinguishing a sequence of random variables from a
  translate of itself}, Ann. Math. Statist. \textbf{36} (1965), 1107--1112.
  \MR{0176509 (31 \#781)}

\bibitem[WZ77]{wheeden-zygmund}
Richard~L. Wheeden and Antoni Zygmund, \emph{Measure and integral}, Marcel
  Dekker Inc., New York, 1977, An introduction to real analysis, Pure and
  Applied Mathematics, Vol. 43. \MR{0492146 (58 \#11295)}

\bibitem[ZS00]{serfling-zuo-notions}
Yijun Zuo and Robert Serfling, \emph{General notions of statistical depth
  function}, Ann. Statist. \textbf{28} (2000), no.~2, 461--482. \MR{1790005
  (2001h:62097)}

\end{thebibliography}
\end{document}